\documentclass{amsart}
\usepackage{setspace}
\usepackage{a4}
\usepackage{amssymb,amsmath,amsthm,latexsym}
\usepackage{amsfonts}
\usepackage{amsfonts}
\usepackage{graphicx}
\usepackage{textcomp}
\usepackage{cite}
\usepackage{enumerate}
\usepackage[mathscr]{euscript}
\usepackage{mathtools}
\newtheorem{theorem}{Theorem}[section]

\newtheorem{definition}[theorem]{Definition}

\newtheorem{problem}[theorem]{Problem}

\title{This is the title}
\usepackage{amssymb}
\usepackage{amssymb}
\usepackage{amsmath}
\usepackage{tikz}
\usepackage{hyperref}
\usepackage{enumerate}
\usepackage{mathtools}
\usepackage{amsmath}
\usepackage{tikz}
\usepackage{amssymb}
\usepackage{amsmath}
\usepackage{tikz}
\usepackage{hyperref}
\usepackage{amssymb}
\usepackage{enumerate}
\usepackage{mathtools}
\usepackage{amsmath}
\usepackage{tikz}
\usepackage{graphicx}
\usepackage{textcomp}
\usetikzlibrary{cd}
\usepackage{fancyhdr}

\begin{document}
	\hrule\hrule\hrule\hrule\hrule
	\vspace{0.3cm}	
	\begin{center}
		{\bf{INDEFINITE HALMOS, EGERVARY AND   SZ.-NAGY  DILATIONS}}\\
		\vspace{0.3cm}
		\hrule\hrule\hrule\hrule\hrule
		\vspace{0.3cm}
		\textbf{K. MAHESH KRISHNA}\\
		Post Doctoral Fellow \\
		Statistics and Mathematics Unit\\
		Indian Statistical Institute, Bangalore Centre\\
		Karnataka 560 059, India\\
		Email: kmaheshak@gmail.com\\
		
		Date: \today
	\end{center}

\hrule
\vspace{0.5cm}
%--------------------------------------
\textbf{Abstract}: Let $\mathcal{M}$ be an indefinite inner product module over a *-ring of characteristic 2. We show that every self-adjoint operator on $\mathcal{M}$ admits Halmos, Egervary and Sz.-Nagy dilations.


\vspace{0.5cm}
\hrule 
%\tableofcontents

\section{Introduction}
In 1950,  Halmos \cite{HALMOS}  made a deep insight into structure theory of operators on Hilbert space by exhibiting any contraction as a part of a unitary. In 1953, Sz.-Nagy \cite{NAGY} showed that Halmos result can be extended to powers of contractions using a unitary operator. In 1963, T. Ando \cite{ANDO} showed that there is a version of Sz.-Nagy dilation for commuting contractions. Combined with spectral theory and theory of (several) complex variables, today, dilation theory of contractions is a rapidly evolving area of research and for a comprehensive look, we refer  \cite{SCHAFFER, NAGYLIFTING, ANDO, PAULSENBOOK, PISIERBOOK, SARASON, AMBROZIEAMULLER, AGLERMCCARTHY, DOUGLAS, NAGYFOIAS, LEVYSHALIT, ARVESON, ORRGUIDED, FRAZHO, FOIASFRAZHO, FOIASFRAZHOGOHBERGKAASHOEK, DURSZTNAGY, POPESCU, BERCOVICI, BHATMUKHERJEE, EGERVARY, BHATTACHARYYA, NAGY1960, PARROTT, DRURY, CRABBDAVIE, MCCARTHYSHALIT, VAROPOULOS, CHOIDAVIDSON}. Started in 1970's, dilations of contractions acting on Lebesgue spaces and  Banach spaces  followed Hilbert space developments \cite{FACKLERGLUCK, AKCOGLUSUCHESTON, STROESCU, NAGELPALM, AKCOGLUKOPP, KERNNAGELPALM, FENDLER}.

In 2021, by identifying essential mechanisms of dilation theory, Bhat,  De and Rakshit \cite{BHATDERAKSHITH} obtained surprising results in the set theory context and vector spaces. In 2022, further study in the context of vector spaces was carried  by Krishna and Johnson \cite{KRISHNAJOHNSON}. We note that another vector space variant is also studied by 
Han, Larson, Liu and Liu \cite{HANLARSONLIULIU}. Recently Krishna introduced the notion of magic contractions and derived Sz.-Nagy dilation for p-adic Hilbert spaces and modules \cite{MAHESH}.

In this paper, we  derive indefinite inner product module  versions of Halmos dilation (Theorem \ref{HD}), Egervary N-dilation (Theorem \ref{ED}), Sz.-Nagy dilation (Theorem \ref{ND}). Our article is highly motivated from the paper of Halmos \cite{HALMOS}, Egervary \cite{EGERVARY}, Schaffer \cite{SCHAFFER}, Sz.-Nagy \cite{NAGY}, Bhat, De and Rakshit \cite{BHATDERAKSHITH}, Krishna and Johnson \cite{KRISHNAJOHNSON} and   Krishna \cite{MAHESH}.

\section{Indefinite Halmos, Egervary and Sz.-Nagy  Dilations}
We are going to use the following notions. A ring $\mathcal{R}$ with an automorphism $*$ which is either identity or of order 2 is called as an *-ring. Throughout the paper we assume that characteristic of ring is 2. 
\begin{definition}\cite{MILNORHUSEMOLLER}\label{PADICDEF}
Let  $\mathcal{V}$ be a module  over $\mathcal{R}$. We say that $\mathcal{V}$ is an indefinite inner product module  (we write IIPM) if there is a map (called as indefinite inner product) $\langle \cdot, \cdot \rangle: \mathcal{V} \times \mathcal{V} \to \mathcal{R}$ satisfying following.
\begin{enumerate}[\upshape (i)]
	\item If $x \in \mathcal{V}$ is such that $\langle x,y \rangle =0$ for all $y \in \mathcal{V}$, then $x=0$.
	\item $\langle x, y \rangle =\langle y, x \rangle^*$ for all $x,y \in \mathcal{V}$.
	\item $\langle \alpha x+y, z \rangle =a \langle x,  z \rangle +\langle y,z \rangle $ for all  $a  \in \mathcal{R}$, for all $x, y, z \in \mathcal{V}$.
\end{enumerate}
\end{definition}
Let $\mathcal{V}$ be a  IIPM and $T:\mathcal{V}\to \mathcal{V}$ be a morphism. We say that $T$ is adjointable if there is a morphism, denoted by $T^*:\mathcal{V}\to \mathcal{V}$ such that $\langle Tx,y\rangle =\langle x,T^*y\rangle$, $\forall x, y \in \mathcal{V}$. Note that (i) in Definition \ref{PADICDEF} says that adjoint, if exists,  is unique. An adjointable morphism $U$ is said to be a unitary if $UU^*=U^*U=I_\mathcal{V}$, the identity operator on $\mathcal{V}$.  An adjointable morphism $P$ is said to be projection if $P^2=P^*=P$. An adjointable morphism $T$ is said to be an isometry if $T^*T=I_\mathcal{V}$. An adjointable morphism $T$ is said to be  self-adjoint  if $T^*=T$. We denote the identity operator on $\mathcal{V}$ by $I_\mathcal{V}$. 

Our first result is the indefinite  Halmos dilation. 
\begin{theorem} (\textbf{Indefinite  Halmos dilation})\label{HD}
    	Let $\mathcal{V}$ be a  IIPM over a *-ring of characteristic 2 and 	$T: \mathcal{V} \to \mathcal{V}$ be a self-adjoint morphism. Then the morphism 
    	\begin{align*}
    		U\coloneqq \begin{pmatrix}
    			T & I_\mathcal{V}+T   \\
    			I_\mathcal{V}+T & T  \\
    		\end{pmatrix}
    	\end{align*}
    	is unitary  on 	$\mathcal{V}\oplus \mathcal{V}$. In other words,
    	\begin{align*}
    		T=P_\mathcal{V}U|_\mathcal{V}, \quad 	T^*=P_\mathcal{V}U^*|_\mathcal{V},
    	\end{align*}
    	where $P_\mathcal{V}:\mathcal{V}\oplus \mathcal{V}\ni (x, y) \mapsto x \in \mathcal{V}$.
    \end{theorem}
\begin{proof}
A direct calculation says that 	
\begin{align*}
	V\coloneqq \begin{pmatrix}
		T & I_\mathcal{V}+T   \\
		I_\mathcal{V}+T & T  \\
	\end{pmatrix}
\end{align*}
is the inverse and adjoint of $U$.	
\end{proof}
Our second result is the indefinite Egervary N-dilation.
\begin{theorem} \label{ED}(\textbf{Indefinite Egervary N-dilation})
Let $\mathcal{V}$ be a IIPM over a *-ring of characteristic 2 and $T:\mathcal{V}\to \mathcal{V}$ be a self-adjoint morphism. Let $N$ be a natural number. Then the morphism
\begin{align*}
U\coloneqq \begin{pmatrix}
T & 0& 0 & \cdots &0&0 & I_\mathcal{V}+T   \\
I_\mathcal{V}+T & 0& 0 & \cdots &0&0& T   \\
0&I_\mathcal{V}&0&\cdots &0&0& 0\\
0&0&I_\mathcal{V}&\cdots &0&0 & 0\\
\vdots &\vdots &\vdots & &\vdots & \vdots &\vdots \\
0&0&0&\cdots &0&0 & 0\\
0&0&0&\cdots &I_\mathcal{V}&0 & 0\\
0&0&0&\cdots &0&I_\mathcal{V} & 0\\
\end{pmatrix}_{(N+1)\times (N+1)}
\end{align*}
is unitary on 	$\oplus_{k=1}^{N+1} \mathcal{V}$ and 
\begin{align}\label{FINITEDILATIONEQUATION}
T^k=P_\mathcal{V}U^k|_\mathcal{V},\quad \forall k=1, \dots, N, \quad (T^*)^k=P_\mathcal{V}(U^*)^k|_\mathcal{V},\quad \forall k=1, \dots, N,
\end{align}
where $P_\mathcal{V}:\oplus_{k=1}^{N+1} \mathcal{V} \ni (x_k)_{k=1}^{N+1} \mapsto x_1 \in \mathcal{V}$. 		
\end{theorem}
\begin{proof}
A direct calculation of power of $U$ gives Equation (\ref{FINITEDILATIONEQUATION}). To complete the proof, now we need show that $U$ is unitary. Define 
\begin{align*}
	V\coloneqq \begin{pmatrix}
		T & I_\mathcal{V}+T& 0 & \cdots &0&0 & 0   \\
		0 & 0& I_\mathcal{V} & \cdots &0&0& 0   \\
		0&0&0&\cdots &0&0& 0\\
		0&0&0&\cdots &0&0 & 0\\
		\vdots &\vdots &\vdots & &\vdots & \vdots &\vdots \\
		0&0&0&\cdots &0&I_\mathcal{V} & 0\\
		0&0&0&\cdots &0&0 & I_\mathcal{V}\\
		I_\mathcal{V}+T&T&0&\cdots &0&0 & 0\\
	\end{pmatrix}_{(N+1)\times (N+1)}.
\end{align*}
Then $UV=VU=I_{\oplus_{k=1}^{N+1} \mathcal{V}}$ and  $U^*=V$.
\end{proof}
  Note that the Equation (\ref{FINITEDILATIONEQUATION}) holds only upto $N$ and not for $N+1$ and higher natural numbers. 
 In the following theorem, given a IIPM $\mathcal{V}$,  $\oplus_{n=-\infty}^{\infty} \mathcal{V}$ is the IIPM defined by 
\begin{align*}
\oplus_{n=-\infty}^{\infty} \mathcal{V}\coloneqq \{ \{x_n\}_{n=-\infty}^\infty, x_n \in \mathcal{V}, \forall n \in \mathbb{Z}, x_n\neq 0 \text{ only for finitely many } n'\text{s}\}
\end{align*}
equipped with  inner product 
\begin{align*}
	\langle \{x_n\}_{n=-\infty}^\infty, \{y_n\}_{n=-\infty}^\infty\rangle \coloneqq \sum_{n=-\infty}^{\infty}\langle x_n, y_n \rangle , \quad \forall \{x_n\}_{n=-\infty}^\infty, \{y_n\}_{n=-\infty}^\infty\in \oplus_{n=-\infty}^{\infty} \mathcal{V}.
\end{align*}
Our third result is the indefinite  Sz.-Nagy dilation.
\begin{theorem}\label{ND} (\textbf{Indefinite Sz.-Nagy dilation})
Let $\mathcal{V}$ be a IIPM  over a *-ring of characteristic 2 and 	$T: \mathcal{V} \to \mathcal{V}$ be a self-adjoint morphism.  Let $U\coloneqq(u_{n,m})_{-\infty \leq n,m\leq \infty}$ be the morphism defined on 
$\oplus_{n=-\infty}^{\infty} \mathcal{V}$ given by  the infinite matrix defined as follows:
\begin{align*}
&u_{0,0}\coloneqq T, \quad u_{0,1}\coloneqq I_\mathcal{V}+T, \quad u_{-1, 0}\coloneqq I_\mathcal{V}+T, \quad u_{-1, 1}\coloneqq T, \\
& u_{n,n+1}\coloneqq I_\mathcal{V}, \quad \forall n \in \mathbb{Z}, n\neq 0,1,  \quad u_{n,m}\coloneqq 0 \quad  \text{ otherwise},
\end{align*}
i.e., 
\begin{align*}
U=\begin{pmatrix}
 &\vdots &\vdots & \vdots & \vdots & \vdots &\vdots & \\
 \cdots & I_\mathcal{V}&0&0&0& 0&0&\cdots &\\
\cdots & 0 & I_\mathcal{V} & 0 & 0&  0&0& \cdots & \\
\cdots & 0 & 0& I_\mathcal{V}+T & T& 0&0&\cdots  & \\
\cdots & 0&0&\boxed{T}&I_\mathcal{V}+T& 0&0&\cdots&\\
\cdots & 0&0&0&0& I_\mathcal{V}& 0&\cdots &\\
\cdots & 0&0&0&0& 0&I_\mathcal{V}&\cdots &\\
 & \vdots &\vdots &\vdots &\vdots  &\vdots & \vdots & \\
\end{pmatrix}_{\infty\times \infty}
\end{align*}
where $T$ is in the $(0,0)$  position (which is boxed), is unitary   on 	$\oplus_{n=-\infty}^{\infty} \mathcal{V}$ and 
\begin{align}\label{INFINITEDILATIONEQUATION}
T^n=P_\mathcal{V}U^n|_\mathcal{V},\quad \forall n\in \mathbb{N}, \quad (T^*)^n=P_\mathcal{V}(U^*)^n|_\mathcal{V},\quad \forall n\in \mathbb{N},
\end{align}
where $P_\mathcal{V}:\oplus_{n=-\infty}^{\infty} \mathcal{V}\ni  (x_n)_{n=-\infty}^{\infty} \mapsto x_0 \in \mathcal{V}$.
\end{theorem}
\begin{proof}
We  get Equation (\ref{INFINITEDILATIONEQUATION}) by calculation of powers of $U$. The matrix   $V\coloneqq(v_{n,m})_{-\infty \leq n,m\leq \infty}$ defined by  
\begin{align*}
	&v_{0,0}\coloneqq T, \quad v_{0,-1}\coloneqq I_\mathcal{V}+T, \quad v_{1, 0}\coloneqq I_\mathcal{V}+T, \quad v_{1, -1}\coloneqq T, \\
	& v_{n,n-1}\coloneqq I_\mathcal{V}, \quad \forall n \in \mathbb{Z}, n\neq 0,1,  \quad v_{n,m}\coloneqq 0 \quad  \text{ otherwise},
\end{align*}
i.e., 
\begin{align*}
V=\begin{pmatrix}
&\vdots &\vdots &\vdots & \vdots & \vdots & \vdots & \\
\cdots&I_\mathcal{V} & 0 & 0& 0 & 0&  0& \cdots & \\
\cdots &0& I_\mathcal{V} & 0& 0 & 0&  0& \cdots & \\
\cdots &0& 0 & I_\mathcal{V}+T& \boxed{T} & 0& 0&\cdots  & \\
\cdots &0& 0&T&I_\mathcal{V}+T&0& 0&\cdots&\\
\cdots &0& 0&0&0&I_\mathcal{V}& 0&\cdots &\\
\cdots &0& 0&0&0&0& I_\mathcal{V}&\cdots &\\
&\vdots &\vdots &\vdots &\vdots &\vdots  & \vdots & \\
\end{pmatrix}_{\infty\times \infty}
\end{align*}
where $T$ is in the $(0.0)$  position (which is boxed), satisfies $UV=VU=I_{\oplus_{n=-\infty}^{\infty} \mathcal{V}}$ and  $U^*=V$. 
\end{proof}
We note that explicit sequential form of $U$ is
\begin{align*}
	U(x_n)_{n=-\infty}^{\infty}=(\dots, x_{-2}, x_{-1}, (I_\mathcal{V}+T)x_0+Tx_1, \boxed{Tx_0+(I_\mathcal{V}+T)x_1}, x_2, x_2, \dots)
\end{align*}
where $Tx_0+(I_\mathcal{V}+T)x_1$ is in the $0$  position (which is boxed)
and $U^*$ is 
\begin{align*}
	U^*(x_n)_{n=-\infty}^{\infty}=(\dots, x_{-3}, x_{-2},  \boxed{(I_\mathcal{V}+T)x_{-1}+Tx_0}, Tx_{-1}+(I_\mathcal{V}+T)x_0, x_1, \dots),
\end{align*}
where $(I_\mathcal{V}+T)x_{-1}+Tx_0$ is in the $0$  position (which is boxed).
We next wish to derive indefinite isometric Sz.-Nagy dilation.
\begin{theorem}\label{ISOSZNAGY} (\textbf{Indefinite isometric Sz.-Nagy dilation})
Let $\mathcal{V}$ be a IIPM over a *-ring of characteristic 2 and $T:\mathcal{V}\to \mathcal{V}$ be a self-adjoint morphism.	Let $U\coloneqq(u_{n,m})_{0\leq n,m\leq \infty}$ be the morphism defined on $\oplus_{n=0}^{\infty} \mathcal{V}$ given by  the infinite matrix defined as follows:
\begin{align*}
	u_{0,0}\coloneqq T, \quad u_{2,1}\coloneqq I_\mathcal{V}+T, \quad 
	 u_{n+1, n}\coloneqq I_\mathcal{V}, \quad \forall n \geq 2,  \quad u_{n,m}\coloneqq 0 \quad  \text{ otherwise},
\end{align*}
i.e., 
\begin{align*}
	U=\begin{pmatrix}
		 &\boxed{T}&0&0&0& 0&0&\cdots &\\
		 & I_\mathcal{V}+T & 0 & 0 & 0&  0&0& \cdots & \\
		 &0 & I_\mathcal{V}& 0& 0& 0&0&\cdots  & \\
		 &0&0&I_\mathcal{V}&0& 0&0&\cdots&\\
		 &0&0&0& I_\mathcal{V}&0 & 0&\cdots &\\
		 &0&0&0&0& I_\mathcal{V}&0&\cdots &\\
		&\vdots&\vdots &\vdots &\vdots  &\vdots & \vdots & \\
	\end{pmatrix}_{\infty\times \infty}
\end{align*}
where $T$ is in the $(0,0)$  position (which is boxed), is isometry    on 	$\oplus_{n=0}^{\infty} \mathcal{V}$ and 
\begin{align}
	T^n=P_\mathcal{V}U^n|_\mathcal{V},\quad \forall n\in \mathbb{N}, \quad (T^*)^n=P_\mathcal{V}(U^*)^n|_\mathcal{V},\quad \forall n\in \mathbb{N},
\end{align}
where $P_\mathcal{V}:\oplus_{n=0}^{\infty} \mathcal{V}\ni  (x_n)_{n=0}^{\infty} \mapsto x_0 \in \mathcal{V}$.
\end{theorem}
\begin{proof}
It suffices to note the adjoint of $U$	is 
	\begin{align*}
		U^*=\begin{pmatrix}
			&\boxed{T}&I_\mathcal{V}+T&0&0& 0&0&\cdots &\\
			& 0 & 0 & I_\mathcal{V} & 0&  0&0& \cdots & \\
			&0 & 0& 0& I_\mathcal{V}& 0&0&\cdots  & \\
			&0&0&0&0& I_\mathcal{V}&0&\cdots&\\
			&0&0&0& 0&0 & I_\mathcal{V}&\cdots &\\
			&0&0&0&0& 0&0&\cdots &\\
			&\vdots&\vdots &\vdots &\vdots  &\vdots & \vdots & \\
		\end{pmatrix}_{\infty\times \infty}
	\end{align*}
where $T$ is in the $(0,0)$  position (which is boxed).
\end{proof}
We now formulate following problems.
\begin{problem}
\textbf{\begin{enumerate}[\upshape(i)]
		\item Whether there is an indefinite Ando dilation? If yes, whether  one can dilate commuting three, four, ... commuting self-adjoint morphisms to commuting unitaries?
		\item Whether there is (a kind of) uniqueness of indefinite  Halmos dilation?
		\item Whether there is a indefinite  intertwining-lifting theorem (commutant lifting theorem)?
	\end{enumerate}}
\end{problem}

  \bibliographystyle{plain}
 \bibliography{reference.bib}

\end{document}